\documentclass[12pt, twoside]{article}
\usepackage{amsmath,amsthm,amssymb}
\usepackage{times}
\usepackage{enumerate}

\def\titlerunning#1{\gdef\titrun{#1}}
\makeatletter
\def\author#1{\gdef\autrun{\def\and{\unskip, }#1}\gdef\@author{#1}}
\def\address#1{{\def\and{\\\hspace*{18pt}}\renewcommand{\thefootnote}{}%
\footnote {#1}}%
\markboth{\autrun}{\titrun}}
\makeatother
\def\email#1{e-mail: #1}
\def\subjclass#1{{\renewcommand{\thefootnote}{}%
\footnote{\emph{Mathematics Subject Classification (2010):} #1}}}
\def\keywords#1{\par\medskip
\noindent\textbf{Keywords.} #1}

\newtheorem{thm}{Theorem}[section]
\newtheorem{cor}[thm]{Corollary}
\newtheorem{lem}[thm]{Lemma}

\newtheorem{prop}[thm]{Proposition}

\newtheorem{letterthm}{Theorem}

\newtheorem{addendum}[letterthm]{Addendum}

\theoremstyle{definition}

\newtheorem{exa}[thm]{Example}

\numberwithin{equation}{section}

\def\G{\Gamma}

\def\wh{\widehat}

\def\Z{\mathbb{Z}}

\def\ns{\vartriangleleft}
\def\<{\langle}
\def\>{\rangle}
\def\a{\alpha}
\def\t{\tau}
\def\L{\Lambda}

\def\ad{{\rm{ad}}}
\def\aut{{\rm{Aut}}}
\def\S{\Sigma}
\def\B{\mathcal{B}}
\def\N{\mathbb{N}}

\def\longhook{{\lhook\joinrel\relbar\joinrel\rightarrow}}

\def\onto{\twoheadrightarrow}

\begin{document}


\baselineskip=17pt



\titlerunning{Strong profinite genus}

\title{The strong profinite genus of a finitely presented group can be infinite} 

\author{Martin R. Bridson} 

\date{Formatted 27 May 2014}

\maketitle

\centerline{\small{\em{For Fritz Grunewald, in memoriam}}}

\address{Martin R.~Bridson: Mathematical Institute,
Andrew Wiles Building, ROQ, Oxford University,
Oxford, OX2 6GG, United Kingdom;  \email{bridson@maths.ox.ac.uk}}



\begin{abstract}{We construct the first examples of finitely-presented, residually-finite groups
$\G$ that contain an infinite sequence of non-isomorphic
finitely-presented subgroups $P_n\hookrightarrow\G$
such that the inclusion maps induce isomorphisms of profinite completions $\wh{P}_n\cong\wh\G$.} 

\end{abstract}

\section{Introduction} 

The profinite completion of a group $\G$ is the inverse
limit of the directed system of
finite quotients of $\G$; it is denoted $\hat \G$. 
If $\G$ is residually finite then the natural
map $\G\to \hat \G$ is injective. In 1970 
Alexander Grothendieck \cite{gro}  
posed the following problem:  let $\G_1$ and $\G_2$ be
residually finite groups and
let $u:\G_1\to\G_2$ be a homomorphism such that the induced map
of profinite completions 
$\hat u :\hat\G_1\to\hat\G_2$ is an isomorphism; if $\G_1$ and $\G_2$ are finitely presented, must
$u$ be an isomorphism?  
This problem was settled in 2004 by Bridson and Grunewald \cite{BG} who showed that $u$ need not
be an isomorphism. 
(The corresponding result for finitely generated groups had been established earlier by Platonov and Tavgen \cite{PT}.) 
There has since
been a considerable amount of work exploring the extent to which   $\G_1$ can differ from $\G_2$,
but the existence of groups of the sort described in the following theorem has remained unknown.

\begin{letterthm}\label{main} There exists a finitely-presented residually-finite group $\G$ and a 
recursive sequence of finitely presented subgroups $u_n:P_n\hookrightarrow\G$ such that each of
the maps
$\wh{u}_n: \wh{P}_n\to\wh{\G}$ is an isomorphism, but $P_m\cong P_n$ if and only if $m=n$.
\end{letterthm}

The analogous result with the $P_n$ finitely generated was proved in Section 9 of \cite{BG} (cf.~\cite{PT}, \cite{BL}). 
The difficulties that prevented us from proving Theorem \ref{main} in \cite{BG} are 
overcome here using two new ideas. First, in place of the results from 
\cite{BBMS} used to establish the finite presentability of fibre
products in \cite{BG}, we use the Effective Asymmetric 1-2-3 Theorem 
proved in \cite{BHMS}; this breaking of symmetry is compatible with
the Platonov-Tavgen criterion for profinite equivalence \cite{PT} as distilled in \cite{BL} and \cite{BG}. The other new idea
is inspired by classical results concerning Nielsen equivalence and is described in two forms, the first of which involves the construction of
a particular type of non-Hopfian group (Section \ref{s:hopf}) and the second of which involves Nielsen equivalence more directly (Section \ref{s:niel}).
In each case, we construct a finitely presented group $Q$ that
admits epimorphisms $p_0:G\to Q$ and  $p_n:\Lambda\to Q\ (n\in\mathbb N)$ so that  the
fibre products $P_n<\G:=G\times \Lambda$ of $(p_0,p_n)$ satisfy Theorem \ref{main}.

Using recent work of Agol \cite{agol} and Wise \cite{wise} (alternatively \cite{HW}), one can arrange for $\G=G\times\Lambda$ to be the fundamental group of a 
special cube complex and hence deduce the following (see Section \ref{s:nilp}).

\begin{addendum}\label{addend}  One can assume that the group $\G$ in Theorem \ref{main} is 
residually torsion-free nilpotent, and each $u_n:P_n\to\G$ induces an isomorphism of 
pro-nilpotent completions.
\end{addendum}

Fritz Grunewald and I tried to prove Theorem \ref{main} when writing \cite{BG} but had insufficient tools at the time.
If Fritz were still alive, the present paper would surely have been a joint one. He is sorely missed.

\section{Asymmetric Fibre Products}

Given two epimorphisms $p_1:\G_1\to Q$ and $p_2:\G_2\to Q$, one has the {\em fibre product} 
$$
P=\{(x,y)\in\G_1\times\G_2 \mid p_1(x)=p_2(y)\}.
$$
The 1-2-3 Theorem of \cite{BBMS} gives conditions under which $P$ is finitely presentable. These
are too restrictive for our purposes but the following
refinement from \cite{BHMS} will serve us well.

\begin{thm}\label{t:123} There exists an algorithm that,
given the following data describing group homomorphisms
$f_i:\G_i\to Q\ (i=1,2)$, will output a finite presentation of
the fibre product $P$ of these maps, together with a
map $P\to\G_1\times\G_2$ defined on the generators, provided that both the $f_i$
are surjective and at least one of the
kernels $\ker f_i$ is finitely generated. (If either of these
conditions fails, the procedure will not halt.) {\bf Input:}
\begin{itemize}
\item A finite presentation $\mathcal Q\equiv \<X\mid R\>$ for $Q$.
\item A finite presentation $\<\underline a^{(i)}\mid \underline r^{(i)}\>$
for $\G_i\ (i=1,2)$.
\item $\forall a\in\underline a^{(i)}$, a word  $\tilde a$ in the free group on $X$
such that $\tilde a=f_i(a)$ in $Q$.
\item A finite set of identity sequences that generates $\pi_2\mathcal Q$
as a $\Z Q$-module.
\end{itemize}
\end{thm}

We shall only need this theorem in the case where $\mathcal Q$ is aspherical, i.e. $\pi_2\mathcal Q=0$,
in which case the algorithm simplifies considerably. The algorithmic nature of the construction is needed
to justify the word ``recursive" in the statement of Theorem \ref{main}.

The above theorem allows us to present fibre products. We shall use it in combination with the
following criterion for proving that the inclusion of certain fibre products induce isomorphisms of
profinite completions.
This criterion is  essentially due to Platonov and Tavgen \cite{PT}. They dealt only with the
case $G_1=G_2$ and $p_1=p_2$, but the distillation of their argument described in Section 5 of \cite{BG} 
applies directly to the asymmetric case.

\begin{thm}\label{PT}
Let $p_1:G_1\to Q$ and $p_2:G_2\to Q$ be epimorphisms and
let $P\subset G_1\times G_2$ be the associated fibre product.
If $G_1$ and $G_2$ are finitely generated, $Q$ has no finite quotients,
and $H_2(Q,\mathbb Z) =0$, then the inclusion $u:P\hookrightarrow
G_1\times G_2$ induces
an isomorphism $\hat u: \hat P\to \hat G_1\times\hat G_2$.
\end{thm}

\section{A Rips construction}

We assume that the reader is familiar with the theory of (Gromov) hyperbolic groups.
We shall use a version of the well-known Rips construction \cite{rips} to construct hyperbolic groups
with controlled properties. In the current setting we need to control the automorphisms of the groups constructed,
 and for this we appeal to the following lemma, the essence of which is taken from \cite{BH}.

\begin{lem}\label{l:BH} Let $\G$ be a torsion-free hyperbolic group 
and let $N\ns\G$ be a non-trivial subgroup that is finitely generated 
and normal. If ${\rm{Out}}(\G)$ is infinite, then $\G/N$ is virtually cyclic.
\end{lem}

\begin{proof} It follows from Rips's theory of group actions on $\mathbb R$-trees
and Paulin's Theorem \cite{paul}, \cite{BSwar} that if ${\rm{Out}}(\G)$ is infinite then
$\G$ acts on a simplicial tree with cyclic arc stabilizers; see \cite{BF} Corollary 1.3. Let $A$ be the stabilizer of an edge in this tree.
Proposition 2.2 of \cite{BH} implies that either
$N$ is contained in $A$, or else $NA$ has finite index in $\G$. The first possibility cannot occur, because it would imply that $N$ was an infinite cyclic normal subgroup of $\G$, and the only torsion-free hyperbolic group with such a subgroup is $\Z$.
Thus $NA$ has finite index in $\G$ and $\G/N$ is commensurable with a quotient of $A$, which is cyclic.
\end{proof}

The original Rips construction was an algorithm that took as input a finite presentation $\mathcal Q$ for a group $Q$ and gave as output a 
small cancellation presentation for a group $\G$ and an epimorphism $\G\to Q$ with finitely generated kernel. There have since been many refinements
of this construction in which extra properties are imposed on $\G$. The most important of these from our point of view is Haglund and Wise's proof \cite{HW} that
one can require $\G$ to be {\em{virtually special}}. (An alternative proof of this can be obtained by combining Wise's results about 
cubulating small cancellation groups \cite{wise} with Agol's proof \cite{agol} that cubulated hyperbolic groups are virtually special.)
For us the key properties of virtually special groups are that they are residually-finite and (what is more) 
that each has a subgroup of finite index that is residually torsion-free-nilpotent \cite{DK}.
We summarize this discussion as follows: 

\begin{prop}\label{rips1}
There exists an algorithm that, given a 
finite group-presentation $\mathcal Q \equiv  \< X\mid R\>$
will construct a finite presentation $\mathcal P \equiv \< X\cup A\mid R'\cup V\>$
for a group $\G$ so that
\begin{enumerate}
\item $\G$ is torsion-free, hyperbolic and residually finite,
\item $N:=\<A\>$ is normal in $\G$,
\item $\G/N$ is isomorphic to the group with presentation $\mathcal Q$.
\item If $Q$ is not virtually cyclic, then ${\rm{Out}}(\G)$ is finite.
\item If $Q$ has no finite quotients, then one may assume that $\G$ is special; in particular it is residually torsion-free-nilpotent.
\end{enumerate} 
\end{prop}

\begin{proof} The only item that is not covered by the preceding discussion and  Lemma \ref{l:BH} is (5), where the phrase ``one may assume" needs
explaining. Rips's original construction gives a short exact sequence $1\to N\to \G \to Q\to 1$ satisfying items (1) to (4) but even with the
work of Agol and Wise in hand one knows only that $\G$ is {\em virtually} special. To remedy this, we pass to a subgroup of finite
index $\G_0<\G$ that is special. Since $Q$ has no finite quotients, $\G_0\to Q$ is still onto. And the kernel, being of finite index in $N$,
is still finitely generated. Thus we may replace $\G$ and $N$ by $\G_0$ and $\G_0\cap N$ preserving properties (1) to (4). 
\end{proof}

\section{Non-Hopfian groups with no finite quotients}\label{s:hopf}

A group $H$ is termed {\em{non-Hopfian}} if there is an epimorphism $H\twoheadrightarrow H$ with non-trivial
kernel. 
Let $$S=\< a, t \mid ta^2t^{-1}=a^3\>.$$  Famously, Baumslag and Solitar \cite{BS} recognised that this group
is non-Hopfian.

\begin{lem}\label{l:S} The given presentation of $S$ is aspherical, $S/\<\!\<t\>\!\>$ is trivial,
and $\psi:a\mapsto a^2, t\mapsto t$ defines an
epimorphism with non-trivial kernel.
\end{lem}

\begin{proof} To see that $\psi$ is onto, observe that $t$ and $a= a^3a^{-2} = ta^2t^{-1}a^{-2}$ are in the image.
Britton's lemma assures us that $c:=[a,tat^{-1}]\in\ker\psi$ is  non-trivial. It is obvious that
$S/\<\!\<t\>\!\>$ is trivial. A standard topological argument shows that the natural presentations of HNN extensions of free groups are 
aspherical (cf. \cite{BG}, p.364).  Alternatively, we can appeal to the fact that 1-relator presentations where the relation is
not a proper power are aspherical.
 \end{proof}

The group $S$ belongs to the family of groups considered in section 4.2 of \cite{BG}, where it is proved that 
a certain amalgamated free product $B=S_1\ast_L S_2$ has no non-trivial finite quotients. Here, $S_1$ and $S_2$
are isomorphic copies of $S$ (with subscripts to distinguish them),  $L$ is a free
group of rank $2$, and the amalgamation makes the identification $c_1=t_2$ and $t_1=c_2$, where $c=[a,tat^{-1}]$, as above.
Thus $B$ admits the following {\em aspherical} presentation:
$$
B = \<a_1, t_1, a_2, t_2 \mid t_1a_1^2t_1^{-1}a_1^{-3},\, t_2a_1^2t_2^{-1}a_2^{-3},\, t_2^{-1}[a_1,t_1a_1t_1^{-1}],\, 
t_1^{-1}[a_2,t_2a_2t_2^{-1}]\>.
$$
The  features of $B$ that we need in this section are the following, which are established in \cite{BG}, p.365. (Other
features will be used in Section \ref{s:niel}.)
A finite presentation of a group is termed {\em{balanced}} if it has the same number of generators as relations.

\begin{lem}\label{l:B} $B$ is an infinite group that has a balanced, aspherical presentation and
 no non-trivial finite quotients. In particular, $B$ is torsion-free and $H_1(B,\Z)=H_2(B,\Z)=0$.
\end{lem}

\subsection{The group we seek}\label{ss:Q}

Let $B$ be a group satisfying the hypotheses of Lemma \ref{l:B}, fix an element of infinite order $b\in B$
and define 
$$
Q= S\ast_\Z B
$$
where the amalgamation indentifies $t\in S$ with $b\in B$.

\begin{prop}\label{p:Q} $Q$ is a non-Hopfian group that has a balanced, aspherical presentation and
 no non-trivial finite quotients. In particular, $Q$ is torsion-free and $H_1(Q,\Z)=H_2(Q,\Z)=0$.
 \end{prop}
 
 \begin{proof} Let $B=\<X\mid R\>$ be a balanced aspherical presentation and let $\beta$ be a word in the generators that
 equals $b^{-1}\in B$. Then
 $$
 \mathcal{Q}\equiv \< a, t, X \mid ta^2t^{-1}a^{-3},\ t\beta,\ R\>
 $$
 is an aspherical balanced presentation for $Q$. In any finite quotient of $Q$, the image of $B$ is trivial, hence the image of $t=b$ 
 is trivial. And since $S$ is in the normal closure of $t$, it too has trivial image. 
 
 To see that $Q$ is non-Hopfian we consider
 the homomorphism $\Psi:Q\to Q$ whose restriction to $S$ is the epimorphism $\psi$ of Lemma \ref{l:S} and whose restriction to $B$
 is the identity: $\Psi$ is well-defined because $\psi(t)=t$ and $S\cap B=\<t\>$; it is onto because $S$ and $B$ lie in the image;
 and it has non-trivial kernel because $\psi$ does.
\end{proof}

\section{Isomorphisms between fibre products}

A subgroup $H < G_1\times G_2$ of a direct product is termed a {\em{sub-direct product}} if the coordinate projections
map it {\em onto} $G_1$ and $G_2$, and it is said to be {\em full} if both of the intersections $H\cap G_i$ are non-trivial.
The fibre product of any pair of epimorphisms $G_1\to Q$ and $G_2\to Q$ is a subdirect product, and it is full provided
both maps have non-trivial kernel. (All subdirect products of $G_1\times G_2$ arise in this way; see \cite{BM} p.362.)

\begin{lem}\label{l:fp} Let $\G_1$ and $\G_2$ be torsion-free, non-elementary hyperbolic groups, let $P, P' < \G_1\times\G_2$ be
full sub-direct products, let $N_i=P\cap\G_i$ and let $N_i'=P'\cap\G_i$. 
Then, every isomorphism $\phi:P\to P'$ maps $N_1\times N_2$ isomorphically onto $N_1'\times N_2'$ (sending the direct summands to
direct summands) and
extends uniquely to an isomorphism $\Phi:\G_1\times\G_2\to \G_1\times\G_2$. 
\end{lem}

\begin{proof} Let $N_i = P\cap \G_i$ and $N_i' = P'\cap \G_i$ and note that these are the kernels of the
coordinate projections restricted to $P$ and $P'$. Note too that $N_i$ and $N_i'$ are normal in $\G_i$ because,
for example, $N_1$ is normal in $P$ and the projection of $P$ onto $\G_1$ fixes $N_1$. A non-trivial normal
subgroup of a torsion-free, non-elementary hyperbolic group contains a non-abelian free group, so the centraliser in $P$ of any
$n\in N_1$ contains such a free group (since it contains $N_2$). On the other hand, elements of $P$ that
do not lie in $N_1\cup N_2$ are of the form $(\gamma_1, \gamma_2)$ with $\gamma_i\in\Gamma_i\smallsetminus\{1\}$,
and non-trivial elements of torsion-free hyperbolic groups have cyclic centralizers. Thus $N_1\cup N_2$
consists of precisely those $x\in P$ with non-abelian centralizer. And
$N_1'\cup N_2' \subset P'$ can be characterised similarly. It follows that every isomorphism $\phi:P\to P'$
sends $N_1\cup N_2$ bijectively to $N_1'\cup N_2'$, and therefore maps $N_1\times N_2$ isomorphically onto $N_1'\times N_2'$.
Further consideration of centralisers shows that $\phi$ sends direct factors to direct factors: 
 either $\phi(N_1)=N_1'$ and $\phi(N_2)=N_2'$ or else $\phi(N_1)=N_2'$ and $\phi(N_2)=N_1'$.

If $\phi(N_1)=N_1'$ and $\phi(N_2)=N_2'$, then the coordinate projections give us natural identifications
$$
P/N_2 = \G_1 = P'/N_2' \text{  and  } P/N_1 = \G_2 = P'/N_1',
$$
via which the maps $pN_2\mapsto \phi(p) N_2'$ and $pN_1\mapsto \phi(p) N_1'$ define the unique isomorphisms $\phi_1:\G_1\to \G_1$
and $\phi_2: \G_2\to\G_2$ such that $\Phi:=(\phi_1,\phi_2)\in {\rm{Aut}}(\G_1)\times {\rm{Aut}}(\G_2)$ restricts to $\phi:P\to P'$.

If $\phi(N_1)=N_2'$ and $\phi(N_2)=N_1'$ then instead we obtain isomorphisms $\tilde\phi_1:\G_1\to \G_2$
and $\tilde\phi_2: \G_2\to\G_1$ such that $\Phi:=\tilde\phi_1\times\tilde\phi_2\in {\rm{Aut}}(\G_1\times \G_2)$ restricts to $\phi:P\to P'$.
\end{proof}

\section{The first proof of Theorem \ref{main}}\label{s:proof}

We shall need the following elementary lemma.

\begin{lem}\label{l:grow} Let $G$ be a group, let $K\ns G$ be a normal subgroup and let $\phi:G\to G$ be an automorphism.
If $K\subsetneq \phi(K)$, then $\phi$ has infinite order in ${\rm{Out}}(G)$.
\end{lem}

\begin{proof} If $\phi^m$ were an inner automorphism for some $m>0$ then, since $K$ is normal, we would have
$\phi^m(K)=K$, whereas $K\subsetneq \phi(K)$ implies $K\subsetneq \phi^m(K)$ .
\end{proof}

We turn to the main argument.
Let $\mathcal{Q}$ be the aspherical presentation described in Proposition \ref{p:Q} and let 
$$1\to N\to \G\overset{\pi_0}\to Q\to 1$$
be the short exact sequence obtained by applying the Rips construction to it. Let $\Psi:Q\to Q$
be the epimorphism described in the proof of Proposition \ref{p:Q} and define $\pi_n = \pi_0\circ\Psi^n$.
Let $P_n< \G\times\G$ be the fibre product of the maps $\pi_0:\G\times\{1\}\to Q$ and $\pi_n:\{1\}\times \G\to Q$.
The kernel of $\pi_0$ is finitely generated, so we have all of the data required to apply Theorem \ref{t:123}.
Thus we obtain, in a recursive manner, finite presentations for the fibre products $P_n$. Theorem \ref{PT}
assures us that the inclusion $P_n\hookrightarrow\G\times\G$ induces an isomorphism of profinite completions.
Thus the following claim completes the proof of Theorem \ref{main}.

\smallskip

\noindent{\em{Claim:}} $P_n\cong P_m$ if and only if $m=n$.

\smallskip

The intersection of $P_n$ with $\G\times\{1\}$ is $\ker\pi_0$ while its intersection with $\{1\}\times\G$
is $\ker\pi_n$. Thus $P_n$ contains the subgroup  $K_n :=  \ker\pi_0\times \ker\pi_n$, which is normal in $\G\times\G$.
Note that $K_n\subsetneq K_m$ if $m>n$. 

Lemma \ref{l:fp} tells us that any isomorphism $\phi: P_n\to P_m$ is the restriction to $P_n$ of an automorphism $\Phi$ of $\G\times\G$.
The automorphism group of $\G\times\G$ contains ${\rm{Aut}}(\G)\times{\rm{Aut}}(\G)$ as a subgroup of index 2, and
Lemma \ref{l:BH} tells us that the group of inner automorphisms has finite index in this. In particular, $\Phi$ has
finite order in the outer automorphism group of $\G\times\G$. But then $\Phi(K_n)=K_m$, by Lemma \ref{l:fp}, which
contradicts Lemma \ref{l:grow} unless $m=n$. This completes the proof of Theorem \ref{main}.
\qed

\section{Pro-nilpotent equivalences}\label{s:nilp}

The {\em{pro-nilpotent}} completion of a group $G$ is the inverse limit of its system of nilpotent quotients; equivalently it is 
the inverse limit of the sequence $G/G_c \to G/G_{c+1}$ where $G_c$ is the $c$-th term of the lower central series of $G$.
If a homomorphism of finitely generated groups  induces an isomorphism of profinite completions, then it induces
an isomorphism of pro-nilpotent completions (\cite{BR}, Proposition 3.2). Thus Addendum \ref{addend} will be proved
if we can arrange for the group $\G$ of the previous section to be residually torsion-free-nilpotent. Proposition \ref{rips1}(5)
assures us that we can do so.

\section{Nielsen Equivalence and $T$-equivalence}\label{s:niel}

The proof of Theorem \ref{main} presented in Section \ref{s:proof} is an implementation of the following naive idea: if one
has a group $Q$ of type $F_3$ and an infinite family of epimorphisms from finitely presented groups $\pi_0:\G\to Q$ and $\pi_n:G\to Q$,
where $\pi_0$ has a non-trivial finitely generated kernel and the $\pi_n$ are ``truly inequivalent" then one expects the fibre products $P_n<\G\times G$ of pairs $(\pi_0,\pi_n)$
to be non-isomorphic. The most direct way in which one might try to implement this strategy is to let $G$ be a free group and take the
$\pi_n$ to be Nielsen-inequivalent choices of generating sets, but this approach is fraught with technical difficulties. In this
section we consider an alternative implementation of the naive strategy that takes up the idea of Nielsen equivalence more directly than our
first proof of Theorem \ref{main}, providing us with different examples.

Let $F$ be a free group with ordered basis $\{x_1,\dots,x_n\}$ and let $G$ be a group. Ordered generating sets
$\Sigma=\{s_1,\dots,s_n\}\subset G$ of cardinality $n$ correspond to epimorphisms $q_\S:F\onto G$;
one defines $q_\S(x_i)=s_i$. The 
automorphism groups ${\rm{Aut}}(F)$ and ${\rm{Aut}}(G)$ act on
the set of such epimorphisms by pre-composition and post-composition, respectively. These actions
commute. By definition, $\Sigma=\{s_1,\dots,s_n\}$ and $\Sigma'=\{s'_1,\dots,s_n'\}$
(or $q_\Sigma$ and $q_{\Sigma'}$) are {\em Nielsen equivalent} if
they lie in the same ${\rm{Aut}}(F)$ orbit, and {\em $T$-equivalent}  if they lie in the same orbit under the
action of ${\rm{Aut}}(F)\times{\rm{Aut}}(G)$. In other words, they are $T$-equivalent if 
there are automorphisms $\phi:F\to F$ and $\psi: G\to G$ making the following diagram commute:
$$
\begin{matrix}
&F&\overset{\phi}\longrightarrow& 
F&\cr
&\ {\bigg\downarrow} {q_\S}&&{\bigg\downarrow} {q_{\S'}}\cr
&G
&\overset{\psi}\longrightarrow&
G
&\cr
\end{matrix}
$$ 
 There is a considerable literature on Nielsen equivalence but
it  is a notoriously difficult invariant to compute and little is known when $n>2$. For 2-generator, 1-relator
groups, the situation is better understood.  

\begin{exa}\label{ex:brunner}
Let $S=\< a, t \mid ta^2t^{-1}=a^3\>$. Because $a= a^3a^{-2} = ta^2t^{-1}a^{-2}$,
for every positive integer $n>0$,
the 2-element set $\S_n=\{t, a^{2^n}\}$ generates $G$. Brunner \cite{brunner} proves that $\S_n$ is
not $T$-equivalent to $\S_m$ if $n\neq m$. 
\end{exa}

In Section \ref{s:hopf} we considered the group
$$
B = \<a_1, t_1, a_2, t_2 \mid t_1a_1^2t_1^{-1}a_1^{-3},\, t_2a_1^2t_2^{-1}a_2^{-3},\, t_2^{-1}[a_1,t_1a_1t_1^{-1}],\, 
t_1^{-1}[a_2,t_2a_2t_2^{-1}]\>,
$$
which is an amalgam of the form $S\ast_ LS$ with $L$ free of rank $2$.
Although Nielsen equivalence behaves well with respect to free products \cite{weid}, it does not
behave well with respect to amalgamated free products, so there is no obvious way of adapting
the generating sets in Example \ref{ex:brunner} so as to produce an infinite sequence of $T$-inequivalent
generating sets for $B$.  To circumvent this problem, we pass from consideration of maps from the
free group of rank $4$ to $B$
to consideration of maps $\Lambda\to B$, where $\Lambda$ is obtained from the free group of rank $4$ taking a trivial HNN extension that distinguishes
a free factor $F$ of rank $2$.

Let $\Lambda = \< \a_1, \t_1, \a_2, \t_2, \zeta \mid [\a_1,\zeta]= [\t_1,\zeta]=1\>$ and let $F=\<\a_1,\t_1\> < \Lambda$.

\begin{lem}\label{l:autL} For every automorphism $\phi:\Lambda\to \Lambda$
there exists $l\in\Lambda$ so that
 $\ad_l\circ\phi$ sends $\zeta$ to $\zeta^{\pm 1}$ and restricts to an automorphism of $C_\L(\zeta) = F\times\<\zeta\>$.
 \end{lem}

\begin{proof}
A simple calculation with HNN normal forms shows that
the only elements $\lambda\in \L$ whose centraliser $C_\L(\lambda)$ contains a non-abelian free group are the conjugates of 
powers of $\zeta$. And $\<\zeta\>$ is maximal among cyclic subgroups of $\Lambda$ (as one can see by abelianising, for example). So $\phi(\zeta) = l\zeta^{\pm 1}l^{-1}$ for some $l\in\Lambda$, which implies that  
$\ad_l\circ\phi$ preserves $C_\L(\zeta) = F\times\<\zeta\>$.
\end{proof}

\begin{lem}\label{l:NoGo}
 Let $Q$ be a group, let $G<Q$ be a subgroup with trivial centralizer,
  let $q, q' : \L\onto Q$ be epimorphisms with kernels $N$ and $N'$, and 
suppose that $q(F)=q'(F)=G$. If $q|_F$ and $q'|_F$  are not $T$-equivalent, 
then there is no automorphism $\phi : \L\to \L$  with $\phi(N)=N'$.
\end{lem}

\begin{proof}  
By definition, $\zeta$ commutes with $F$ and we are assuming that $q(F)=q'(F)=G$  has trivial centralizer,
so $q(\zeta)=q'(\zeta)=1$ and both $q$ and $q'$ factor through the retraction
$\rho:{F\times\{\zeta\}}\to F$.

Towards a contradiction, suppose that there is an automorphism $\phi:\L\to\L$ such that $\phi(N)=N'$.
Then $q(x)\mapsto q'(\phi(x))$ defines an automorphism  $\overline\phi:Q\to Q$.
Lemma \ref{l:autL} tells us that there is an element $l\in \L$ so that the inner automorphism ${\rm{ad}}_l$
conjugates $\phi(F\times\<\zeta\>)$ to $ F\times\<\zeta\>$. Thus, we obtain the following commutative diagram
$$  
\begin{matrix}
&F&\longhook&
F\times\<\zeta\>&\overset{{\phi}}\longrightarrow&
\phi(F\times\<\zeta\>)&\overset{\ad_l}\longrightarrow&
F\times\<\zeta\>&\overset{\rho}\longrightarrow&
F\cr
&\ {\bigg\downarrow} {q}&&{\bigg\downarrow} {q}
&& {\bigg\downarrow} {q'}&&{\bigg\downarrow} {q'}
&& {\bigg\downarrow} {q'}
\cr
&G&\overset{{\rm{id}}}\longrightarrow&
G&\overset{\overline{\phi}}\longrightarrow&
\overline{\phi}(G)&\overset{\ad_{q'(l)}}\longrightarrow&
G&\overset{{\rm{id}}}\longrightarrow&
G.
\end{matrix}
$$
This diagram
shows  that  $q|_F$ and $q'|_F$ are $T$-equivalent, contrary to hypothesis.
\end{proof} 

The vertex groups in the decomposition 
$B=S\ast_ LS$ are centreless, and each is a maximal elliptic subgroup (in the sense of Bass-Serre theory), therefore
each has trivial centralizer in $B$. Thus we may apply the preceding lemma to maps $\Lambda\to B$ with
the first factor $S$ in the role of $G$. 

Working with the presentations of $\Lambda$ and $B$ displayed above, we define
$ q_n: \Lambda\to B$ by setting 
$$ q_n(\alpha_i)=a_i^{2^n},\ \ q_n(\tau_i)=t_i,\ \ q_n(\zeta)=1.$$

\begin{cor}\label{c:noWay} If $n\neq m$, there is no automorphism $\phi:\Lambda\to\Lambda$
such that $\phi(\ker q_n) = \ker q_m$.
\end{cor}

\begin{proof} 
The discussion in Example \ref{ex:brunner} shows that each $q_n$ is surjective and that the restriction of
$q_n$ to $F=\<\alpha_1,\tau_1\>$ is T-equivalent to the restriction of $q_m$ if and only if $m=n$.
\end{proof}
   
\section{A second proof of Theorem \ref{main}}
\def\B{\mathcal{B}}
\def\N{\mathbb{N}}

We apply the Rips construction  to the finite presentation of $B$ given above to obtain a short exact sequence $1\to N \to \G\overset{\pi_0}\to B\to 1$
satifying the conditions of Proposition \ref{rips1}. As $\G$ is special, it is a subgroup of a right-angled Artin group (RAAG). $\Lambda$
itself is a RAAG, and hence $\G\times\Lambda$ is special; in particular it is residually-finite and residually torsion-free-nilpotent.
$B$ is given by a finite aspherical presentation and it has no non-trivial finite quotients so, as in Section \ref{s:proof}, 
we will be done if we can prove that the fibre products
$$
\Pi_n = \{(x,y) \mid \pi_0(x)=q_n(y)\} < \G\times\Lambda
$$ associated to the epimorphisms $q_n:\Lambda \to B$ from Corollary \ref{c:noWay}
have the property that $\Pi_n\not\cong\Pi_m$ if $n\neq m$.

By arguing with centralizers, as in the proof of Lemma \ref{l:fp}, one sees that
every  isomorphism $\phi:\Pi_n\to \Pi_m$ is the restriction of an 
ambient automorphism  $(\phi_1,\phi_2)\in\aut(\Gamma)\times\aut(\Lambda)$ with
$$\phi_2(\Pi_n\cap \Lambda) = \Pi_m\cap\Lambda.$$
(Here we have taken account of the fact that $\G$ and $\Lambda$ are not isomorphic.)
But $\Pi_i\cap\Lambda = \ker q_i$, so this contradicts Corollary \ref{c:noWay} unless $n=m$. \qed
        
\noindent\textit{Acknowledgments.}
The author is supported by the EPSRC and by a Wolfson Research Merit Award from the
Royal Society. He thanks these organisations for their support.

\end{document}